\documentclass{article}

\usepackage{authblk}

\usepackage{amsmath}
\usepackage{amsthm}
\usepackage{amsfonts}
\usepackage{amssymb}
\usepackage[all]{xy}
\usepackage{url}

\DeclareMathSymbol{\rightrightarrows}  {\mathrel}{AMSa}{"13}

\def\sk{\operatorname{sk}}

\def\max{\operatorname{max}}

\def\Im{\operatorname{Im}}

\catcode`\@=11
\def\varholim@#1#2{\mathop{\vtop{\ialign{##\crcr
        \hfil$#1\m@th\operator@font holim$\hfil\crcr
 \noalign{\nointerlineskip\kern\ex@}#2#1\crcr
 \noalign{\nointerlineskip\kern-\ex@}\crcr}}}}
\def\hocolim{\mathpalette\varholim@\rightarrowfill@} 
\def\hoinvlim{\mathpalette\varholim@\leftarrowfill@}
\catcode`\@=\active

\newtheorem{theorem}{Theorem}%[section]
\newtheorem{lemma}[theorem]{Lemma}
\newtheorem{corollary}[theorem]{Corollary}

\theoremstyle{definition}

\newtheorem{example}[theorem]{Example}

\newtheorem{remark}[theorem]{Remark}

\begin{document}

\title{\bf Local data structures}
\author{{\bf J.F. Jardine}\thanks{Supported by NSERC.}}

%\affil{\small Department of Mathematics\\University of Western Ontario\\
%  London, Ontario, Canada
%}
\affil{jardine@uwo.ca}

\maketitle

\begin{abstract}
Local data structures are systems of neighbourhoods within data sets. Specifications of neighbourhoods can arise in multiple ways, for example, from global geometric structure (stellar charts), combinatorial structure (weighted graphs), desired computational outcomes (natural language processing), or sampling. These examples are discussed, in the context of a theory of neighbourhoods.

This theory is a step towards understanding clustering for large data sets. These clusters can only be approximated in practice, but approximations can be constructed from neighbourhoods via patching arguments that are derived from the Healy-McInnes UMAP construction. The patching arguments are enabled by changing the theoretical basis for data set structure, from metric spaces to extended pseudo metric spaces.
\end{abstract}

%62R40, 55U10, 68T09.   62H30, 68R10

\section*{Introduction}

This paper is a preliminary discussion of local structures for
large data sets.

Potential objects of study include subsets $\mathcal{U} \subset \mathbb{R}^{N}$, where the data set $\mathcal{U}$ (or ``universe'') is essentially infinite, meaning that $\mathcal{U}$ is too large to analyze with available computational devices.

Alternatively, there may not be a metric space structure on the data set $\mathcal{U}$. Such objects $\mathcal{U}$ can arise as vertices of large weighted graphs $\Gamma$, which could describe data transfers that occur during a time interval. Other examples arise in the ``bag of words'' model natural language processing, which model has a combinatorial structure that is not graph theoretic. 

There could, finally, be no apparent geometric or combinatorial structure for $\mathcal{U}$, and its structure near a point may have to be approximated (or learned) by iterated sampling.
\medskip

In general, one wants to break up a data set $\mathcal{U}$ into smaller computable pieces $N$ that cover $\mathcal{U}$ in the sense that every $x \in \mathcal{U}$ is in some neighbourhood $N$, in the hope/expectation that analyses of the neighbourhoods $N$ can be assembled to a full or at least useful partial analysis of the universal data set $\mathcal{U}$. 
This is essentially the approach taken by the mapper algorithm \cite{CMS} (see Remark \ref{rem 12} below), and it can make perfect sense for clustering at relatively small distance scales.

The elements of a neighbourhood $N$ should be close to $x$ in some sense, but one has to address the question of how to find such neighbourhoods in a sea of data $\mathcal{U}$.
If there is no prior information about the structure or genesis of $\mathcal{U}$, the phrase ``close to $x$'' may not have much meaning. In good cases, there is information about local geometric or combinatorial structures that allows one to get started.
\medskip

Most generally, a neighbourhood $N$ of a point $x$ in a data set $\mathcal{U}$ is a suitably sized subset of $\mathcal{U}$ which contains $x$. If $\mathcal{U}$ is a metric space (or an extended pseudo metric space) then $N$ has a diameter $s(N)$, which is the maximum distance $d(x,y)$ for $y \in N$.

The inclusion $N \subset \mathcal{U}$ determines an inclusion of Vietoris-Rips complexes $V(N) \subset V(\mathcal{U})$.

If every $x \in \mathcal{U}$ has a specific choice of neighbourhood $N_{x}$, as in Section 3, then the collection of all such neighbourhoods determines an inclusion of filtered complexes
\begin{equation*}
  N(\mathcal{U}):=\cup_{x \in \mathcal{U}}\ V(N_{x}) \subset V(\mathcal{U}),
\end{equation*}
which complexes are filtered by distance in the usual way.
I say that $N(\mathcal{U})$ is the {\it neighbourhood complex} that is defined by the family of neighbourhoods $N=\{ N_{x} \}$.

The neighbourhood complex $V(N)$ is the mapper complex for the covering $V(N_{x}) \subset V(\mathcal{U})$ of the global Vietoris-Rips complex $V(\mathcal{U})$, as in \cite{CMS}.

Every element $y \ne x$ in a neighbourhood $N_{x}$ determines a ray
\begin{equation*}
  \{x,y\} \subset N_{x} \subset \mathcal{U},
\end{equation*}
and the collection of such rays determines
a filtered subcomplex
\begin{equation*}
  R(N_{x}) = \vee_{y \ne x}\ V(\{x,y\}) \subset V(N_{x}).
\end{equation*}
Taking the union
\begin{equation*}
  R(\mathcal{U}) = \cup_{x \in \mathcal{U}}\ R(N_{x}) \subset V(\mathcal{U})
\end{equation*}
defines the {\it ray subcomplex} $R(\mathcal{U})$, which is a subcomplex of both $V(\mathcal{U}$ and $N(\mathcal{U})$. 
\medskip

The ray subcomplex $R(\mathcal{U})$  is a filtered (or weighted) graph.

If the neighbourhoods $N_{x}$ consist of sets of $k$-nearest neighbours for the points of $\mathcal{U}$, then the ray subcomplex $R(\mathcal{U})$ is the $k$-nearest neighbours graph, which is a well-studied object. The $k$-nearest neighbours graph is used to construct the UMAP graph of \cite{HMc-2020}, \cite{UMAP-stab}, \cite{github}.
\medskip

The inclusions
\begin{equation*}
  R(\mathcal{U}) \subset N(\mathcal{U}) \subset V(\mathcal{U})
\end{equation*}
of filtered complexes induce surjections
\begin{equation*}
  \pi_{0}R_{s}(\mathcal{U}) \to \pi_{0}N_{s}(\mathcal{U}) \to \pi_{0}V_{s}(\mathcal{U})
\end{equation*}
for distance parameters $s$, which are analyzed in special cases in Sections 3 and 4.
There are good comparison results for finite $s$ for bounded neighbourhoods, which is the subject of Section 4. See Lemma \ref{lem 14},
Lemma \ref{lem 15}, Lemma \ref{lem 16} and Lemma \ref{lem 20}.

In that setting, the neighbourhood complex $N_{s}(\mathcal{U})$ for bounded neighbourhoods has the same $1$-skeleton as the global Vietoris-Rips complex $V_{s}(\mathcal{U})$ at small distance scales $s$, which makes the neighbourhood complex $N_{s}(\mathcal{U})$ a good approximation of $V_{s}(\mathcal{U})$ for clustering for such $s$.

At higher distance scales, the clusters of the ray complex $R_{s}(\mathcal{U})$ coincide with those of the neighbourhood complex $N_{s}(\mathcal{U})$. The outcome is that, for clustering, the neighbourhood complex $N(\mathcal{U})$ is a bridge between the ray complex $R(\mathcal{U})$ (a UMAP-like object) and the full Vietoris-Rips complex $V(\mathcal{U})$.

The basic ideas and constructions of this paper appear in the Sections 2 and 3, along with a discussion of the relationship between neighbourhoods and sequences of nearest neighbours.
With a view to potential applications (as in Section 5), we generally assume that $\mathcal{U}$ is an extended pseudo metric space, or an ep-metric space.  The basic ideas around ep-metric spaces are summarized in Section 1.
%\medskip

Subsequent results and calculations are determined by choices of neighbourhoods, which choices vary with the geometric or combinatorial structures of specific examples. 

The definitions and results of Sections 4, 6 and 7 are based on naive examples (or thought experiments) that motivate and illustrate these ideas:
\medskip

\noindent
1)\ 
The Gaia Archive $\mathcal{U}$ is a database of roughly a billion stars in the Milky Way. The raw data for the Archive is a set of scans that has been collected by the Gaia Space Observatory spacecraft, starting in 2014.
The scans return high resolution photometric and spectral data for stars within small apertures, and so the archive is constructed from an assembly of local data. The positions of the stars in the archive relative to the Sun are determined, after repeated observations and much computation.

These positions can be expressed as a function $p: \mathcal{U} \to \mathbb{R}^{3}$ that determines the members of the Archive $\mathcal{U}$ uniquely.
The position function $p$ is a type of dimension reduction. In the language of the mapper construction, it is a filter function \cite{CMS}.

From observation, if $x$ is a star in the archive $\mathcal{U}$, then there is a neighbourhood $N_{x} \subset \mathcal{U}$ of stars close to $x$ such that $N_{x}$ has a computable number of elements. We could insist that $N_{x}$ is a bounded neighbourhood, in that it has a bounded radius $s(N_{x})$ and it contains at most $k$ elements for some choice of integer $k$. 

This is an explicitly geometric example, which is closely aligned with methods that are presented in Section 4.
\medskip

\noindent
2)\
For some data sets, there is a graph structure $\Gamma$ with no apparent ambient metric space.

For example, a collection of data transfers between computer accounts within a (short enough) time interval can be given the structure of a sparse directed weighted graph, as in Example \ref{ex 30} below. The number of bytes transmitted by a transfer is its weight.

The vertices of these graphs have low valence. One knows all of the transfers $e: x \leftrightarrow y$ for each account $x$, and from this one builds a computable neighbourhood $N_{k}(x)$ of accounts which are separated from $x$ by at most $k$ transfer steps (or hops).

One needs a way of assigning weights $d(x,y)$ to the various $y \in N_{k}(x)$.
Starting with an account $x$, one could expect that the accounts $y$ with which it does the most ``business'' are the closest to $x$.
The elements $y$ of $N_{k}(x)$ which are closest to $x$ are defined ``inversely'' by the sum $\Sigma(x,y)$ of all weights of directed edge paths between $x$ and $y$. Then the distance  $d(x,y)$ can be defined by
\begin{equation*}
  d(x,y)= e^{-\Sigma(x,y)}
\end{equation*}
between $x$ and $y$ for each $y \in N_{k}(x)$.
\medskip

From the data of neighbourhoods and weights, the Healy-McInnes UMAP machine generates a global ep-metric $D$ on the set $Z$ of vertices of the graph $\Gamma$, with clusters given by the directed set $\pi_{0}V(Z,D)$, or equivalently (Theorem \ref{th 27}) by the directed set $\pi_{0}R(N)$ arising from the rays of the various neighbourhoods $N_{k}(x)$.

The point, ultimately, is that one uses the graph structure to find computable weighted neighbourhoods $N_{k}(x) $ for all vertices $x$ of a sparse weighted directed graph $\Gamma$. These local structures then patch together to define a global ep-metric on the full set of vertices of $\Gamma$, along with cluster constructions.

These ideas appear in Section 6. In broad outline, they apply
equally well to all sparse weighted graphs.
\medskip

There is a fundamental idea in play here: the UMAP construction creates global space-level structure and cluster computations from local information given by weighted neighbourhoods, with or without the existence of an ambient metric.

This observation is applied repeatedly in examples that are displayed here. We specify neighbourhoods with weights, and then feed these neighbourhoods to general machinery.

The relevant theoretical features of the UMAP construction are summarized in Section 5.  That section contains an alternate presentation of the UMAP graph, which is constructed by patching together rays without invoking most of the standard methods of UMAP --- see Theorem \ref{th 27}.
 \medskip

 \noindent
 3)\ Section 7 is a discussion of neighbourhoods of words in the ``continuous bag of words'' model from natural language processing (NLP). With such neighbourhoods in hand (and with appropriate definitions of weights), one again uses UMAP methods to construct an ep-metric space structure on the set of words $\mathcal{L}$ that of a corpus.
 
The methods of Section 7 extend to any finite set of strings of data elements, in which a local metric can be defined by proximity within strings.
 \medskip

 In the examples displayed so far, the local nature of a data set varies within a given geometric or combinatorial structure. These structures are in part determined by desired computational outcomes, and they are the starting points for calculations.
 
 One could, finally, be presented with a very large cloud of points $\mathcal{U}$ with an ep-metric space structure, but with no other information, from which one wants to approximate (or discover) a neighbourhood $N_{x}$ for a given point $x \in \mathcal{U}$.

 There seems to be no choice in such a case but to apply brute force methods that are based on repeated random sampling, with the goal of learning a description of a neighbourhood, or  ``$k$-complete'' neighbourhood $N_{x}$ for $x$. A potential method for doing so is described in Section 8.

 The $k$-complete neighbourhoods of this paper (see Sections 2 and 4) are strongly related to sets of $k$-nearest neighbours for a point $x$, but have the benefit of being uniquely defined, and are therefore easier to manipulate theoretically. Of course, the positive integer $k$ must be specified up front. 

 \bigskip
%\vfill\eject

\tableofcontents

\section{Extended pseudo metric spaces}

An {\it extended pseudo-metric space} $(X,d)$, here called an {\it ep-metric space}, is a set $X$ together with a function $d:X \times X \to [0,\infty]$ such that the following conditions hold:
  \begin{itemize}
  \item[1)] $d(x,x)=0$,
  \item[2)] $d(x,y) = d(y,x)$,
  \item[3)] $d(x,z) \leq d(x,y) + d(y,z)$.
  \end{itemize}
  There is no condition that $d(x,y)=0$ implies $x$ and $y$ coincide --- this is where the adjective ``pseudo'' comes from, and the gadget is ``extended'' because we allow infinite distance.

  A metric space $(X,d)$ is an ep-metric space for which $d(x,y)=0$ implies $x=y$, and all distances $d(x,y)$ are finite.

  There is a category $\mathbf{ep-met}$ of ep-metric spaces, with morphisms $f:(X,d) \to (Y,d')$ given by functions $f:X \to Y$ which are non-expanding in the sense that $d'(f(x),f(y)) \leq d(x,y)$ for all $x,y \in X$.

  The category $\mathbf{ep-met}$ is a cocomplete in the sense that it has all small colimits.

  In effect, the coproduct $\sqcup_{i}\ (X_{i},d{i})$ is the  disjoint union set $\sqcup_{i}\ X_{i}$, equipped with the ep-metric $d$ defined by
  \begin{equation*}
    d(x,y) = \begin{cases}
      d_{i}(x,y) & \text{if $x,y \in X_{i}$ for some $i$, and} \\
      \infty & \text{otherwise.}
    \end{cases}
  \end{equation*}

  Coequalizers are constructed from a quotient function. Suppose that $(X,d)$ is an ep-metric space and that $p: X \to Y$ is a surjective function. Then $Y$ has an ep-metric $D$ such that for any pair $z,w \in Y$,
  \begin{equation*}
    D(z,w) = \inf_{P}\ \sum d(x_{i},y_{i}),
  \end{equation*}
  where each ``path'' $P$ consists of pairs of points $\{x_{i},y_{i}\},\ i \leq n$ in $X$ such that $z=p(x_{0})$, $w = p(y_{n})$ and $p(y_{i}) = p(x_{i+1})$ for $i \leq n-1$. The function $p$ defines a map $p: (X,d) \to (Y,D)$ of ep-metric spaces that has the universal property of quotients.
  \medskip

  \begin{example}
    Suppose that $(X,d)$ and $(X,d')$ are ep-metric spaces having the same set of oelements $X$. Then the amalgamation (wedge) $(X,d) \vee (X,d')$ in the ep-metric space category is an ep-metric space structure on $X$ with
    \begin{equation*}
      D(z,w) = \inf_{P}\ \sum D(x_{i},x_{i+1}),
    \end{equation*}
    where each path $P$ is a string of elements $z=x_{0},x_{i}, \dots ,x_{n}=w$ of $X$ and
    \begin{equation*}
      D(x_{i},y_{i}) = \min\ \{d(x_{i},x_{i+1}),d'(x_{i},x_{i+1})\}.
      \end{equation*}
    \end{example}
  
%\medskip
Each finite ep-metric space $\mathcal{U}$ has a family of Vietoris-Rips complexes $V_{s}(\mathcal{U})$, which are parameterized by distance $s$. Explicitly, $V_{s}(\mathcal{U})$ is the abstract simplicial complex (or poset) whose simplices are the finite subsets $\sigma=\{x_{0}, \dots ,x_{k}\}$ of $\mathcal{U}$ such that $d(x_{i},x_{j}) \leq s$. The simplex $\sigma$ is a $k$-simplex, and it has cardinality $k+1$.

  As in the standard case, there is an ascending family of complexes
  \begin{equation*}
    V_{s}(\mathcal{U}) \subset V_{t}(\mathcal{U}),\ s \leq t,
  \end{equation*}
  with $\mathcal{U} = V_{0}(\mathcal{U})$ (discrete complex on the set $\mathcal{U}$).

  The limiting object $V_{\infty}(\mathcal{U})$ is a simplex $\Delta^{\mathcal{U}}$ with vertices $\mathcal{U}$, but it is not the case that $V_{\infty}(\mathcal{U})$ is a union of the subobjects $V_{s}(\mathcal{U})$ with $s$ finite. Write
  \begin{equation*}
    V_{<\infty}(\mathcal{U}) = \cup_{s < \infty}\ V_{s}(\mathcal{U}).
  \end{equation*}

  The simplicial set $V_{< \infty}(\mathcal{U})$ is a finite disjoint union of contractible components.

  \section{Neighbourhoods}
    
Suppose that $(Z,d)$ is a finite ep-metric space and that $x \in Z$.

    In all of the following,
\begin{equation*}
  Z(x,s) = \{y \in Z\ \vert\ d(x,y) \leq s\}
  \end{equation*}
is the closed ball of radius $s$ in $Z$ that is centred at $x$.

A {\bf neighbourhood} $N$ of $x$ is a subset 
$N$  of $Z$ with $x \in N$ and $d(x,y) < \infty$.
\medskip

    A neighbourhood $N$ aquires an ep-metric space structure from $Z$, and defines a filtered subcomplex $V(N) \subset V(Z)$ of the Vietoris-Rips complex $V(Z)$.

    The {\bf radius} $s(N)$ of the neighbourhood $N$ is defined by
    \begin{equation*}
      s(N) = \max_{y \in N}\ d(x,y).
    \end{equation*}
    Then $s(N) < \infty$ by assumption.
     
    The neighbourhood $N$ is said to be {\bf complete} if $N=Z(x,s_{N})$. 
\medskip

A neighbourhood $N$ of $x$ is a set of {\bf nearest neighbours} if $d(x,z) \geq s(N)$ for all $z \in Z-N$. If $N = \{x,x_{1}, \dots ,x_{k}\}$ is a set of nearest neighbours (i.e. with cardinality $k+1$), then $N$ is a set of {\bf $k$-nearest neighbours}.

A nearest neighbour $y$ for $x$ with $d(x,y) < \infty$ can be identified with a neighbourhood $N = \{x,y\}$ of nearest neighbours. This means that $d(x,y) \leq d(x,z)$ for all $z \in Z-\{x\}$. The distance $d(x,y)$ could be $0$ in general.

Every complete neighbourhood $N = Z(x,s_{N})$ is a set of nearest neighbours for $x$, and is a set of $n$-nearest neighbours, where $n = \vert N \vert - 1$.

\begin{lemma}\label{lem 2}
Suppose that $N = \{x,x_{1}, \dots ,x_{k}\}$ is a set of nearest neighbours for $x$, and that the $x_{i}$ are ordered such that
\begin{equation*}
  d(x,x_{1}) \leq d(x,x_{2}) \leq \dots \leq d(x,x_{k}).
\end{equation*}
Then $x_{i}$ is a nearest neighbour of $x$ in the subset $Z-\{x_{1}, \dots ,x_{i-1}\}$.
\end{lemma}

\begin{proof}
  We have
  \begin{equation*}
d(x,x_{i}) \leq d(x,x_{i+1}) \leq \dots \leq d(x,x_{k}) \leq d(x,z)
  \end{equation*}
  for all $z$ outside of $N$. It follows that $d(x,x_{i})) \leq d(x,w)$ for all $w \in Z-\{x_{1}, \dots x_{i-1}\}$.
    \end{proof}

\begin{lemma}\label{lem 3}
Suppose that the neighbourhood $N$ is a set of nearest neighbours for $x$ and $z \in X-N$ is chosen such that $d(x,z) < \infty$ and $d(x,z) \leq d(x,v)$ for all $v \in Z-N$. Then the set $N \cup \{z\}$ is a set of nearest neighbours for $x$.
\end{lemma}

\begin{proof}
The radius $s_{z}$ of $N \cup \{z\}$ is $d(x,z)$. Choose $v \in Z - (N \cup \{z\})$. Then $s(N) \leq d(x,v)$, and $d(x,z) \leq d(x,v)$ by the minimality of $d(x,z)$. It follows that $s(N \cup \{z\}) \leq d(x,v)$.
  \end{proof}

\begin{remark}
Applying Lemma \ref{lem 3} inductively gives nearest neighbourhoods $N$ of $x$ of all possible finite cardinalities $\vert N \vert$ with $\vert N \vert \leq \vert Z \vert$.
\end{remark}

There is a function $d_{x}: Z \to [0,\infty]$ with $d_{x}(y) = d(x,y)$. A nearest neighbour for $x$ is an element $z \in Z -\{x\}$ such that $d_{x}(z) < \infty$ and $d_{x}(z)$ is minimal.

For such an element $z$, write $s = d_{x}(z)$. Then $s$ is the minimum finite value of the image $d_{x}(Z)$, and $z \in Z_{x}(s)$, where
  \begin{equation*}
    Z_{x}(s) = d_{x}^{-1}(s)
  \end{equation*}
  is the fibre (pre-image) of $d_{x}$ over $s$.

  \begin{lemma}\label{lem 5}
  Suppose that $N$ is a set of nearest neighbours for $x$, and suppose that $\{s_{1}, \dots ,s_{p}\}$ is the set of elements of the image $d_{x}(N)$, with $s_{1} < \dots s_{p}$. Then $\{s_{1}, \dots ,s_{p}\}$ is a set of smallest finite elements of $d_{x}(Z)$, and
\begin{equation*}
  N = Z_{x}(s_{1}) \cup \dots \cup Z_{x}(s_{p-1}) \sqcup F
\end{equation*}
where $F \subset Z_{x}(s_{p})$.
  \end{lemma}

  \begin{proof}
This is proved by induction on $\vert N\vert$, using Lemma \ref{lem 2} and Lemma \ref{lem 3}.
  \end{proof}

  If the neighbourhood $N=\{x,x_{1}, \dots ,x_{k}\}$ is a set of nearest neighbours of $x$ with
  \begin{equation*}
    d(x,x_{1}) \leq \dots \leq d(x,x_{k}),
  \end{equation*}
  one says that $(x_{1},x_{2}, \dots ,x_{k})$ is a {\bf sequence of $k$-nearest neighbours} for $x$.
 
\begin{lemma}\label{lem 6}
  Suppose that $\{ y_{1}, \dots ,y_{k}\}$ is a set of distinct elements of $X-\{x\}$ with
\begin{equation*}
  d(x,y_{1}) \leq d(x,y_{2}) \leq \dots \leq d(x,y_{k}) < \infty.
\end{equation*}
If $(x_{1}, \dots ,x_{k})$ is a sequence of $k$-nearest neighbours for $x$, then $d(x,x_{i}) \leq d(x,y_{i})$ for $1 \leq i \leq k$.
\end{lemma}

\begin{proof}
  $d(x,x_{1}) \leq d(x,y_{1})$, since $x_{1}$ is a nearest neighbour.
  \medskip

Suppose that $d(x,x_{i}) \leq d(x,y_{i})$ for $i \leq r$. Then
\smallskip

\noindent
  1)\ If $d(x,x_{r}) < d(x,y_{r+1})$ then $d(x,x_{r+1}) \leq d(x,y_{r+1})$ by minimality.
  \smallskip

  \noindent
  2)\ If $d(x,x_{r}) = d(x,y_{r+1})$ then $y_{r+1}$ is a nearest neighbour of $x$ in $Z-\{x_{1},\dots,x_{r}\}$, and so $d(x,x_{r+1}) = d(x,y_{r+1})$.
\end{proof}

\begin{corollary}
  Suppose that $(x_{1}, \dots ,x_{k})$ and $(y_{1}, \dots ,y_{k})$ are sequences of $k$-nearest neighbours for $x$. Then $d(x,x_{i}) = d(x,y_{i})$ for all $i$.
\end{corollary}

\begin{corollary}
Suppose that $W$ is a finite ep-metric space, and the inclusion $Z \subset W$ induces an ep-metric structure on the subset $Z$. Supppose that $x \in Z$. Suppose that $(w_{1}, \dots ,w_{k})$ and $(z_{1}, \dots ,z_{k})$ are sequences of $k$-nearest neighbours for $x$ in $W$ and $Z$, respectively. Then $d(x,w_{i}) \leq d(x,z_{i})$ for $1 \leq i \leq k$.
  \end{corollary}

\begin{lemma}\label{lem 9}
  Suppose that $(x_{1}, \dots ,x_{k})$ is a sequence of nearest neighbours for $x$ in $Z$, and that $\{y_{1}, \dots ,y_{k}\}$ is a sequence of distinct elements of $Z$ with $d(x,y_{1}) \leq \dots \leq d(x,y_{k}) < \infty$.

  If $d(x,x_{i}) = d(y,y_{i})$ for all $i$, then $(y_{1}, \dots ,y_{k})$ is a sequence of nearest neighbours for $x$.
  \end{lemma}

\begin{proof}
  $d(x,y_{1}) = d(x,x_{1}) \leq d(x,z)$ for all $z$, so that $y_{1}$ is a nearest neighbour for $x$ in $Z$.
  \medskip
  
  Inductively, suppose that $\{y_{1}, \dots ,y_{i}\}$ is a set of nearest neighbours for $x$.

  Suppose that $d(x,x_{i}) = d(x,x_{i+1})$. Then $d(x,y_{i})=d(x,y_{i+1})$, and so $\{y_{1},\dots,y_{i+1}\}$ is a set of nearest neighbours.
  
  If $d(x,x_{i}) < d(x,x_{i+1})$, then $\{y_{1}, \dots ,y_{i}\} = \{x_{1}, \dots ,x_{i}\}$ by comparing fibres $Z_{x}(s)$, so that $y_{i+1}$ is the nearest neighbour of $x$ in $Z - \{y_{1}, \dots ,y_{i}\}$. 
\end{proof}

We close this section with a discussion of $k$-complete neighbourhoods.
\medskip

The image of the distance function $d_{x}: Z \to [0,\infty]$  has the form
\begin{equation*}
  \Im(d_{x}) = \{s_{1},s_{2}, \dots \},
\end{equation*}
where there are strict inequalities $s_{i} < s_{i+1}$ for all $i$. The data set
$Z$ is a disjoint union of non-empty fibres of $d_{x}$: 
\begin{equation*}
  Z = p_{x}^{-1}(s_{1}) \sqcup p_{x}^{-1}(s_{2}) \sqcup \dots = 
  Z_{x}(s_{1}) \sqcup Z_{x}(s_{2}) \sqcup \dots.
\end{equation*}

For each $s_{i}$, there is a unique complete neighbourhood $Z(x,s_{i})$ of $x$,
with
\begin{equation*}
  Z(x,s_{i}) = d_{x}^{-1}(s_{1}) \sqcup \dots \sqcup d_{x}^{-1}(s_{i}).
\end{equation*}
The complete neighbourhoods of $x$ form a finite ascending tower
\begin{equation*}
  Z(x,s_{1}) \subset Z(x,s_{2}) \subset Z(x,s_{3}) \subset \dots
  \end{equation*}
Any complete neighbourhood $N$ with $Z(x,s_{i}) \subsetneqq N$ must have strictly greater radius $s_{N} > s_{i}$.

Suppose that $k$ is a positive integer and that $\vert Z \vert \geq k$. Then there is a smallest number $i$ such that $\vert Z(x,s_{i}) \vert \geq k$. In this case, the neighbourhood $Z(x,s_{i})$ is $k$-complete.

Alternatively, the $k$-complete neighbourhood $N$ of $x$ is the smallest complete neighbourhood such that $\vert N \vert \geq k$.

The element $x \in Z$ has a unique $k$-complete neighbourhood $N$ in $Z$, provided that $\vert Z \vert \geq k$. The $k$-complete neighbourhood $N$ is a well defined object, while there may be multiple sets of $k$-nearest neighbours of $x$

\begin{lemma}\label{lem 10}
Suppose that $Z_{1},Z_{2} \subset \mathcal{U}$, and that $x \in Z_{i}$. Suppose that $N_{i} \subset Z_{i}$ is the $k$-complete neighbourhood of $x$ in $Z_{i}$, and suppose that $N$ is the $k$-complete neighbourhood of $x$ in $Z = Z_{1} \cup Z_{2}$. Then $N \subset N_{1} \cup N_{2}$.
\end{lemma}

\begin{proof}
  The set $N$ is a set of nearest neighbours for $x$ in $Z_{1} \cup Z_{2}$, and $Z_{i} \cap N$ is a set of nearest neighbours for $x$ in $Z_{i}$.

  In effect, if $z \in Z_{i}$ is not in $Z_{i} \cap N$ then $z$ is not in $N$, so that $d(x,z) \geq s_{N}$, while $s_{N} \geq s_{Z_{i} \cap N}$.

  If there is an $s < s_{N}$ such that $\vert Z_{i}(x,s) \vert \geq k$, then $\vert Z(x,s) \vert \geq k$ for $s < s_{N}$, and so $N=Z(x,s_{N})$ is not $k$-complete. It follows that $\vert Z_{i}(x,s) \vert < k$ for $s < s_{N}$, and so $Z_{i}(x,s_{N}) \subset N_{i}$.

  Thus, $N \subset N_{1} \cup N_{2}$, as claimed
  \end{proof}

Lemma \ref{lem 10} leads to a method of approximating $k$-complete neighbourhoods for a point $x$ in a very large data set $\mathcal{U}$.

In effect, if $Z_{i} \subset \mathcal{U}, 1 \leq i \leq p$ is a collection of subsets of $\mathcal{U}$ with $x \in Z_{i}$, and if $N_{i} \subset Z_{i}$ is a $k$-complete neighbourhood of $x$ in $Z_{i}$, then the $k$-complete neighboourhood $N$ of $x$ in $Z_{1} \cup \dots \cup Z_{p}$ is the $k$-complete neighbourhood of $x$ in the much smaller object $N_{1} \cup \dots \cup N_{p}$.   

\section{Topological constructions}

Suppose that $(Z,d)$ is a finite ep-metric spac. 

      Suppose given a set of neighbourhoods $N_{x}$ for each $x \in Z$. Recall that the neighbourhood $N_{x}$ has a diameter $s(N_{x}) < \infty$.
       
       Each neighbourhood $N_{x}$ determines a filtered subcomplex $V(N_{x}) \subset V(Z)$ of the Vietoris-Rips complex $V(Z)$.
 
       The inclusions $\{x,y\} \subset V(N_{x})$, $y \in N_{x} -\{x\}$, induce filtered simplicial complex maps
\begin{equation}\label{eq 1}
   R(N_{x}):= \vee_{y}\ \Delta^{1}_{\geq s} \subset V(N_{x}) \subset V(Z).
\end{equation}
The copies of $\Delta^{1}$ are defined by rays $\{x,y\}$ of weights $s$.

\begin{remark}\label{rem 11}
More properly, if the ray $\{x,y\}$ has weight $t=d(x,y)$, then the corresponding $1$-simplex of $R(N_{x})$ is the filtered simplex $\Delta^{1}_{\geq t}$ such that
\begin{equation*}
  (\Delta^{1}_{\geq t})_{s} =
    \begin{cases}
      \emptyset & \text{if $s < t$, and} \\
      \Delta^{1} & \text{if $s \geq t$.}
    \end{cases}
\end{equation*}

It is better, sometimes, to say that $R(N_{x})$ is {\bf covered} by simplices $\Delta^{1}_{\geq s}$ corresponding to rays $\{x,y\}$ of weight $s$. This simply reflects the fact that the obvious map
\begin{equation*}
  \sqcup_{y}\ \Delta^{1}_{\geq s} \to \vee_{y}\ \Delta^{1}_{\geq s} = R(N_{x})
\end{equation*}
is an epimorphism of filtered complexes.
\end{remark}

The full union
\begin{equation*}
  R(N) = \cup_{x}\ R(N_{x}) \subset V(Z)
\end{equation*}
is the {\bf ray subcomplex} of $V(Z)$, for the collection of neighbourhoods $N = \{N_{x}\}$.

The ray subcomplex $R(N)$ is a filtered (or weighted) graph. If the neighbourhoods $N_{x}$ consist of $k$-nearest neighbours, then $R(N)$ is the $k$-nearest neighbours (kNN) graph.
\medskip
       
The neighbourhoods $N_{x}$ generate an abstract simplicial complex $V(N) \subset V(Z)$ whose simplices are the subsets $\sigma \subset N_{x}$ of the various neighbourhoods $N_{x}$. The resulting filtered simplicial complex can be written
\begin{equation*}
  V(N) = \cup_{x} \ V(N_{x}) \subset V(Z).
\end{equation*}
The subcomplex $V(N)$ of $V(Z)$ is called the {\bf neighbourhood complex}.
\medskip

The inclusions $R(N_{x}) \subset V(N_{x})$ induce an inclusion $R(N) \subset V(N)$, so we have inclusions
\begin{equation}\label{eq 2}
  R(N) \subset V(N) \subset V(Z)
\end{equation}
of filtered complexes, with corresponding inclusions
\begin{equation}\label{eq 3}
  R_{s}(N) \subset V_{s}(N) \subset V_{s}(Z)
\end{equation}
of the various filtration stages.

The induced functions
\begin{equation*}
  \pi_{0}R_{s}(N) \to \pi_{0}V_{s}(N) \to \pi_{0}V_{s}(Z)
\end{equation*}
in path components (or clusters)
are surjective for all parameters $t$, since all complexes have the same vertex set, namely $Z$.

\begin{remark}\label{rem 12}
  The neighbourhood complex $V(N) = \cup_{x}\ V(N_{x})$ is covered by the subcomplexes $V(N_{x})$, in the sense that there is a surjection
\begin{equation*}
  \bigsqcup_{x}\ V(N_{x}) \to V(N).
\end{equation*}
  This covering has an associated \v{C}ech resolution, and there is a natural coequalizer
  \begin{equation*}
    \bigsqcup_{x,y}\ V_{s}(N_{x}) \cap V_{s}(N_{y}) \rightrightarrows \bigsqcup_{x}\ V_{s}(N_{x}) \to V_{s}(N)
  \end{equation*}
  in simplicial sets, where $s$ is the distance parameter. The path component functor preserves colimits, so there is a coequalizer
  \begin{equation*}
    \bigsqcup_{x,y}\ \pi_{0}(V_{s}(N_{x}) \cap V_{s}(N_{y})) \rightrightarrows \bigsqcup_{x}\ \pi_{0}V_{s}(N_{x}) \to \pi_{0}V_{s}(N)
  \end{equation*}
  in diagrams of sets, or clusters.

  The directed set $\pi_{0}V_{s}(N)$ is the cluster object given by the mapper construction for the covering of $Z$ by the family of neighbourhoods $\{N_{x}\}$ \cite{CMS}.    
  \end{remark}

\begin{lemma}\label{lem 13}
Suppose that $x,y \in Z$. There is a path from $x$ to $y$ in $R(N)$ if and only if there is a sequence of elements
\begin{equation*}
  x=x_{0}, x_{1}, \dots ,x_{r}=y
\end{equation*}
and neighbourhoods $N_{x_{i}}$ of $x_{i}$, such that $N_{x_{i}} \cap N_{x_{i+1}} \ne \emptyset$ for all $i$.
\end{lemma}

\begin{proof}
Suppose that 
\begin{equation*}
  x=z_{0}, \dots ,z_{p}=y
\end{equation*}
is a sequence of points such that $x_{i+1} \in N_{x_{i}}$ or $x_{i} \in N_{x_{i+1}}$ for neighbourhoods $N_{x_{i}}$ and $N_{x_{i+1}}$ of $x_{i}$ and $x_{i+1}$, respectively.
If $x_{i+1} \in N_{x_{i}}$ then $N_{x_{i}} \cap N_{x_{i+1}} \ne \emptyset$. Similarly, if $x_{i} \in N_{x_{i+1}}$ then $N_{x_{i}} \cap N_{x_{i+1}} \ne \emptyset$.

Suppose, conversely, that $v \in N_{x_{i}} \cap N_{x_{i+1}}$. Then there is an edge $x_{i} \to v$ in $N_{x_{i}}$ and an edge $x_{i+1} \to v$ in $N_{x_{i+1}}$, so that there is a path
\begin{equation*}
  x_{i} \to v \leftarrow x_{i+1}
\end{equation*}
through neighbourhoods.
\end{proof}

By definition, the ray complex $R(N)$ is a filtered subcomplex of $V(Z)$. The subcomplex $R_{s}(N) \subset V_{s}(Z)$ is generated by rays $\{x,y\}$ with $d(x,y) \leq s$.
\medskip

We have the following analog of Lemma \ref{lem 13}:

\begin{lemma}\label{lem 14}
Suppose that $x,y \in Z$. For each parameter value $s$, there is a path from $x$ to $y$ in $R_{s}(N)$ if and only if there is a sequence of elements
\begin{equation*}
  x=x_{0}, x_{1}, \dots ,x_{r}=y
\end{equation*}
and  neighbourhoods $N_{x_{i}}$ of $x_{i}$, such that $(N_{x_{i}})_{s} \cap (N_{x_{i+1}})_{s} \ne \emptyset$ for all $i$.
\end{lemma}

\section{Bounded neighbourhoods}

\subsection{$k$-bounded neighbourhoods}

      In some examples (such as stellar charts), it is natural that neighbourhoods $N$ of $x$ have bounded cardinality and radius: $\vert N \vert \leq k+1$ for some $k$ and $s(N) \leq S$, with both $k$ and $S$ fixed.

      From this point of view, for a fixed $x$, the {\bf $k$-bounded neighbourhoods} $N$ of $x$ are the subsets of $Z(x,S)$ which contain $x$ and have at most $k+1$ elements. Again, $Z(x,S)$ is the ball of radius $S$ in $Z$, which is centred on $x$.

      We assume that $k \geq 1$ henceforth. 
      \medskip
      
A point $x$ can have more than one $k$-bounded neighbourhood. The
$k$-bounded neighbourhoods of $x$ are ordered by inclusion, and the family has maximal elements. We have the following:
  \begin{itemize}
   \item[1)]  The maximal $k$-bounded neighbourhboods $N \subset Z(x,S)$ either have cardinality $k+1$ or satisfy $N = Z(x,S)$.
     \smallskip
     
   \item[2)] All sets $N$ of $k$-nearest neighbours with $s(N) \leq S$ are maximal.
     \smallskip

\item[3)] If $N = \{x\}$ is maximal, then $x$ is an isolated point for the parameter $S$.
\end{itemize}

   The corresponding neighbourhood complex $V(k-N)$ is the filtered subcomplex of $V(Z)$ that is generated by the subobjects $V(N)$ for all $k$-bounded neighbourhoods $N$ of all $x$, and $R(k-N)$ is the  associated ray subcomplex. As in (\ref{eq 2}), we have a sequence of inclusions
   \begin{equation*}
     R(k-N) \subset V(k-N) \subset V(Z).
   \end{equation*}

If $t \leq S$ and $\sigma = \{x_{0}, \dots ,x_{n}\}$ is an $n$-simplex of $V(Z)_{t}$ with $n \leq k$, then $\sigma$ is a $k$-bounded neighbourhood of $x_{0}$. In effect, $\sigma$ has at most $k+1$ elements, of maximal distance $t \leq S$ from $x_{0}$. It follows that $V_{t}(k-N)_{n} = V_{t}(Z)_{n}$ for $n \leq k$, or that $\sk_{k}V_{t}(k-N) = \sk_{k}V_{t}(Z)$. In particular, $\sk_{1}V_{t}(k-N) = \sk_{1}V_{t}(Z)$ since $k \geq 1$, and so the simplicial sets $V_{t}(k-N)$ and $V_{t}(Z)$ have the same path components.

We have shown the following:

\begin{lemma}\label{lem 15}
  Suppose that $t \leq S$, and construct the neighbourhood complex $V(k-N)$ from $k$-bounded neighbourhoods as above. Then the function
  \begin{equation*}
\pi_{0}V_{t}(k-N) \to \pi_{0}V_{t}(Z)
    \end{equation*}
is a bijection.
  \end{lemma}

Suppose that $t \geq S$, and that $\{x,y\}$ is a $1$-simplex of $V_{t}(k-N)$. Then $\{x,y\} \subset N$ for a $k$-bounded neighbourhood $N$ of some $z$. Further, $d(z,x) \leq s(N)$ and $d(z,y) \leq s(N)$, so that $d(z,x),d(z,y) \leq s(N) \leq S \leq t$. It follows that there is a path
\begin{equation*}
  x \leftarrow z \to y
\end{equation*}
in $R_{t}(k-N)$, and so the function
\begin{equation*}
\pi_{0}R_{t}(k-N) \to \pi_{0}V_{t}(k-N)
\end{equation*}
is a bijection.

We have proved

\begin{lemma}\label{lem 16}
  Suppose that $t \geq S$. Then the induced function
\begin{equation*}
\pi_{0}R_{t}(k-N) \to \pi_{0}V_{t}(k-N)
\end{equation*}
is a bijection.
  \end{lemma}

Write
\begin{equation*}
  R(k-N) = \cup_{t}\ R_{t}(k-N).
  \end{equation*}
Then the map
\begin{equation*}
  \pi_{0}R_{t}(k-N) \to \pi_{0}R(k-N)
  \end{equation*}
is a bijection for $t \geq S$, because $R_{t}(k-N) = R(k-N)$ in that range.

We therefore have the following:
\begin{corollary}
 The functions
  \begin{equation*}
    \pi_{0}R(k-N) \leftarrow \pi_{0}R_{t}(k-N) \to \pi_{0}V_{t}(k-N)
  \end{equation*}
  are bijections for all $t \geq S$.
  \end{corollary}

Write
\begin{equation*}
 N(x) = \cup_{N}\ V(N)
\end{equation*}
in $V(Z)$, where the union is indexed over all $k$-bounded neighbourhoods $N$ of $x$. Let $R(x) \subset N(x)$ be the associated ray subcomplex.
    We have the inclusions
    \begin{equation*}
      VR(x) \subset VN(x) \subset V(Z(x,S)).
    \end{equation*}

    Suppose that $t \leq S$ and
    $n +1\leq k$. Suppose that $\sigma = \{x_{0}, \dots ,x_{n}\}$ is a non-degenerate $n$-simplex of $V(Z(x,S))_{t}$.
    Write
    \begin{equation*}
      \sigma_{x} = \{x,x_{0}, \dots ,x_{n}\}.
    \end{equation*}
    Then $\vert \sigma_{x} \vert \leq k+1$, so that $\sigma_{x}$ is a $k$-bounded neighbourhood of $x$, and so $\sigma = d_{0}\sigma_{x}$ is in the image of the composite
\begin{equation*}
  V(\sigma_{x})_{t} \to VN(x)_{t} \to V(Z(x,S))_{t}.
\end{equation*}

It follows that $\sk_{k-1}V_{t}N(x) = \sk_{k-1}V_{t}(Z(x,S))$ for $t \leq S$.
In particular, the map
\begin{equation*}
  \pi_{0}V_{t}N(x) = \pi_{0}V_{t}(Z(x,S))
\end{equation*}
  is a bijection if $k \geq 2$ and $t \leq S$.

Suppose that $t \geq S$, and that $y$ and $z$ are vertices of $N(x)$. Then $d(x,y), d(x,z) \leq S \leq t$, and it follows that the map
\begin{equation*}
  \ast = \pi_{0}R(x)_{t} \to \pi_{0} N(x)_{t}
\end{equation*}
is a bijection.

Every $y \in Z(x,S)$ is a member of a $k$-bounded neighbourhood $\{x,y\}$ since $k \geq 1$. It follows that the maps
\begin{equation*}
  \pi_{0}R(x)_{t} \to \pi_{0}N(x)_{t} \to \pi_{0}V_{t}Z(x,S)
\end{equation*}
are surjective.

We have proved:

\begin{lemma}\label{lem 18}
  Suppose that the complexes $N(x)$ and $R(x)$ are defined as above. Suppose that $k \geq 2$. Then we have the following:
  \begin{itemize}
  \item[1)] If $t \leq S$ then the map $\pi_{0}N(x)_{t} \to \pi_{0}V_{t}(Z(x,s))$ is a bijection.
  \item[2)] If $t \geq S$ then the maps
\begin{equation*}
  \ast =\pi_{0}R(x)_{t} \to \pi_{0}N(x)_{t} \to \pi_{0}V(Z(x,S))_{t}
\end{equation*}
  are bijections.
\end{itemize}
  \end{lemma}

\subsection{Complete neighbourhoods}

Suppose that $Z$ is a finite ep-metric space, and that each $x \in Z$ has a fixed complete neighbourhood $N_{x}=Z(x,r_{x})$. 
Form the associated filtered complexes
\begin{equation*}
R(N) \subset V(N) \subset V(Z),
\end{equation*}
for $Z$ and the system of neighbourhoods $N = \{N_{x}\}$.

\begin{example}\label{ex 19}
Suppose that $S>0$ is a fixed distance parameter and $k > 1$ is a fixed integer.

Say that a neighbourhood $N_{x}$ of $x \in Z$ is {\bf complete $k$-bounded} if $N_{x}$ has the form
\begin{equation*}
  N_{x} = Z(x,s_{x}) \cap Z(x,S),
\end{equation*}
where $Z(x,s_{x})$ is the unique $k$-complete neighbourhood of $x$ (see Section 2).

There are two possibilities:\ $N_{x} = Z(x,s_{x})$, in which case $N_{x}$ is $k$-complete, or $N_{x}=Z(x,S)$ and $\vert Z(x,S) \vert < k$. In either case, the neighbourhood $N_{x}$ is uniquely determined and is complete.
\smallskip

The use of {\it complete} $k$-bounded neighbourhoods gives a different perspective for the stellar chart example. For a fixed (and appropriate) distance $S$ and positive integer $k$, the complete $k$-bounded neighbourhoods $N_{x}$ of stars $x$ in a globular cluster would be $k$-complete neighbourhoods of small radius, while stars in an outer spiral arm are more likely to have neighbourhoods $N_{x}$ of smaller cardinality.
\end{example}

\begin{lemma}\label{lem 20}
  Suppose that $Z$ is a finite ep-metric space, and that each $x \in Z$ has a fixed complete neighbourhood $N_{x}=Z(x,r_{x})$.
 \begin{itemize}
\item[1)] Suppose that $t \leq r_{x}$ for all $x$. Then the functions
\begin{equation*}
\pi_{0}R_{t}(N) \to \pi_{0}V_{t}(N) \to \pi_{0}V_{t}(Z)
\end{equation*}
are bijections.
\item[2)] Suppose that $t \geq r_{x}$ for all $x$. Then the map
\begin{equation*}
\pi_{0}R_{t}(N) \to \pi_{0}V_{t}(N)
\end{equation*}
is a bijection.
\item[3)] Suppose that $t \geq S \geq r_{x}$ for all $x$. Then the map
\begin{equation*}
\pi_{0}R_{S}(N) \to \pi_{0}R_{t}(N)
\end{equation*}
is a bijection.
\end{itemize}
\end{lemma}

\begin{proof}
For 1), suppose that $\{x,y\}$ is a $1$-simplex of length $t$. Then $y \in N_{x}$ since $t \leq r_{x}$, and $\{x,y\}$ is a ray of $N_{x}$. It follows that there are equalities of $1$-skeleta
\begin{equation*}
\sk_{1}R_{t}(N) = \sk_{1}V_{t}(N) = \sk_{1}V_{t}(Z),
\end{equation*}
and the statement follows.

For statement 2), suppose that $\{x,y\}$ is a $1$-simplex of $V_{t}(N_{z}) \subset V_{t}(N)$. Then there are $1$-simplices $x \leftarrow z \rightarrow y$ in $V_{t}(N_{z})$ since $t \geq r_{z}$. This is true for all $z$, and it follows that the function
\begin{equation*}
\pi_{0}R_{t}(N) \to \pi_{0}V_{t}(N)
\end{equation*}
is a bijection.

To prove statement 3), every ray ($1$-simplex) $\{x,y\}$ of $R(N)$ has length $\leq S$, so that $R_{S}(N) = R_{t}(N)$.
\end{proof}

\begin{corollary}\label{cor 21}
Suppose that $t \geq S \geq r_{x}$ for all $x \in Z$. Then the inclusion $V_{S}(N) \subset V_{t}(N)$ of neighbourhood complexes induces a bijection
\begin{equation*}
\pi_{0}V_{S}(N) \to \pi_{0}V_{t}(N).
\end{equation*}
\end{corollary}

\begin{proof}
The Corollary follows from statements 2) and 3) of the Lemma \ref{lem 20}
\end{proof}

\begin{remark}\label{rem 22}
  Suppose that $Z'$ is the subset of elements $x \in Z$ such that $Z(x,r_{x}) = \{x\}$, and let $Z"=Z-Z'$. Then
  \begin{itemize}
    \item[1)]
      $V_{t}(Z) = Z' \sqcup V_{t}(Z'')$,
    \item[2)] $V(N)_{t} = Z' \sqcup V(N'')_{t}$,
      \item[3)]
        $R(N)_{t}=  Z' \sqcup R(N'')_{t}$,
        \end{itemize}
  for $t \leq r_{z}$, all $z$, where $Z'$ is a discrete set. Here,
\begin{equation*}
  Z'' = \cup_{y \in Z''}\ Z(y,r_{y}),
\end{equation*}
  and $N''$ is the system of neighbourhoods $Z(y,n_{y})$ for $y \in Z''$.
\end{remark}

    \section{The UMAP construction}

One starts with a neighbourhood $N_{x}$ for each vertex $x$ of a data set $Z$, with positive weights $d(x,y)$ for each $y \in N_{x}-\{x\}$. The subset $\{x,y\}$ for such a $y$ is said to be a ray.

The weight $d(x,y)$ defines an ep-metric space structure on the set $\{x,y\}$.
Form the ep-metric space
\begin{equation*}
  Z_{x}= \vee_{y \in N(x)-\{x\}}\ \{x,y\},
  \end{equation*}
from the rays $\{x,y\}$, for each $x \in Z$. This structure is extended to an ep-metric space structure $(Z,D_{x})$ on the full set of vertices $Z$ of $\Gamma$, by setting
\begin{equation*}
  (Z,D_{x}) = (\sqcup_{z \in Z-Z_{x}}\ \{z\}) \sqcup Z_{x}
\end{equation*}
in ep-metric spaces.

The ep-metric space
\begin{equation*}
  (Z,D) = \vee_{x \in Z}\ (Z,D_{x})
\end{equation*}
and the UMAP complex
\begin{equation*}
  V(Z,N) = \vee_{x \in Z}\ V(Z,D_{x})
\end{equation*}
are formed by amalgamating along vertices (elements of $Z$), in ep-metric spaces and filtered complexes, respectively.

It is crucial, for these ep-metric space constructions, to know that the category of ep-metric spaces is cocomplete --- see Section 1.
\medskip

The following excision statement for path components is Lemma 2 of \cite{UMAP-stab}:

\begin{theorem}\label{th 23}
The canonical map $V(Z,N) \to V(Z,D)$ induces isomorphisms
\begin{equation*}
  \pi_{0}V(Z,N)_{s} \xrightarrow{\cong} \pi_{0}V(Z,D)_{s}
\end{equation*}
for $s$ finite.
\end{theorem}

Theorem \ref{th 23} is proved by observing that distances in $(Z,D)$ are computed from paths through neighbourhoods $N_{x}$.
\medskip

We shall need the following local computation:

\begin{lemma}\label{lem 24}
Suppose that $y \in N(x)-\{x\}$ defines the ray $\{x,y\}$. Then
\begin{equation*}
  d(x,y)=D_{x}(x,y).
\end{equation*}
in $Z_{x}$.
\end{lemma}

\begin{proof}
The number $D_{x}(x,y)$ is the minimum of all sums $\sum_{j}\ d(x_{j},x_{j+1})$, for paths
\begin{equation*}
  P:\ x=x_{0},x_{1},\dots, x_{p}=y
\end{equation*}
through rays in $Z_{x}$. The ray $\{x_{p-1},y\}$ must be the ray $\{x,y\}$, so that
\begin{equation*}
  D_{x}(x,y) \geq \sum_{j}\ d(x_{j},x_{j+1}) \geq d(x,y).
\end{equation*}
  The subobject $\{x,y\}$ is a ray, so that $\{x,y\}$ is a path, and $D_{x}(x,y) \leq d(x,y)$.
\end{proof}

Each $y \in N_{x} - \{x\}$ determines an inclusion of filtered complexes
\begin{equation*}
  V(\{x,y\}) \subset V(Z_{x},D_{x}) \subset V(Z,D).
\end{equation*}
The simplicial set $V_{t}(\{x,y\}$ consists of vertices $\{x,y\}$ for $t < d_{x}(x,y)$, and has $1$-simplices
\begin{equation*}
\{x\} \subset \{x,y\} \supset \{y\}
\end{equation*}
for $t \geq d_{x}(x,y)$.

\begin{remark}
Recall that $V(\{x,y\})$ is the barycentric subdivison of a filtered $1$-simplex that would be defined by imposing a total order on the set $\{ x,y \}$.
\end{remark}

Suppose that $\{x,u\}$ is a ray in $N_{x}$ and that $\{y,v\}$ is a ray of $N_{y}$, and consider the composite monomorphisms
\begin{equation*}
  V(\{x,u\}) \subset V(Z_{x}) \subset V(Z),\quad V(\{y,v\}) \subset V(Z_{y}) \subset V(Z).
  \end{equation*}
Suppose that $d_{x}(x,u) \leq d_{y}(y,v)$.

Generally, $V(X) = BP(X)$, where $P(X)$ is a poset of generating simplices. In the case at hand, therere is a pullback diagram
\begin{equation*}
\xymatrix{
BP(\{x,u\} \cap \{y,v\}) \ar[r] \ar[d] & BP(\{y,v\}) \ar[d] \\
BP(\{x,u\}) \ar[r] & BP(Z)
}
\end{equation*}

The intersection $\{x,u\} \cap \{y,v\}$ is at most a $2$-element set.
If $\{x,u\} \cap \{y,v\} = \emptyset$ the pullback is empty, and if
 $\{x,u\} \cap \{y,v\}$ is a point the pullback is a point. 

If $\{x,u\} \cap \{y,v\}$ is a $2$-element set, then $\{x,u\} = \{y,v\}$,  
and there is a commutative diagram
\begin{equation*}
\xymatrix{
BP(\{x,u\}) \ar[dr] && BP(\{y,v\} \ar[ll]_{\theta} \ar[dl]\\
& BP(Z)
}
\end{equation*}
where $\theta$ ``reduces weight''. It follows, in this case, that
there is a pullback
\begin{equation}\label{eq 4}
\xymatrix{
BP(\{y,v\}) \ar[r]^{1} \ar[d]_{\theta} & BP(\{y,v\}) \ar[d] \\
BP(\{x,u\}) \ar[r] & BP(Z)
}
\end{equation}

The ray complex $R(N_{x}) \subset V(Z_{x})$ is the wedge of rays
\begin{equation*}
  R(N_{x}) = \vee_{y \in N_{x}-\{x\}}\ V(\{x,y\}).
\end{equation*}
The filtered complex monomorphisms
\begin{equation*}
  \phi_{x}: R(N_{x}) \to V(Z_{x},D_{x}) \to V(Z,D), 
\end{equation*}
together define a monomorphism
\begin{equation*}
  \phi: R(N) = \cup_{x\in Z}\ R(N_{x}) \subset V(Z),
\end{equation*}
and we say that the union $R(N)$ is the ray subcomplex of $V(Z)$. The ray complex $R(N)$ is a weighted graph.

The ray complex $R(N)$ is a union of (or is covered by) filtered subcomplexes $V(\{x,y\})$, which are defined by rays $\{x,y\}$ and their weights $d(x,y)$. The intersections (pullbacks)
\begin{equation*}
\xymatrix{
BP(\{y,v\} \cap \{y,v\}) \ar[r] \ar[d] & BP(\{y,v\}) \ar[d] \\
BP(\{x,u\}) \ar[r] & R(N)
}
\end{equation*}
are constructed in $V(Z)$ as above, since $R(N) \subset V(Z)$ is a monomorphism. It follows that the ray complex $R(N)$ is a union of rays, with possible adjustments of weights in intersections, as in the pullback diagram (\ref{eq 4}).

\begin{remark}
The present description of the ray complex $R(N)$ is independent of distances in the space $(Z,D)$. It generalizes the description of the ray complex that appears in Section 3, which uses a fixed ambient ep-metric.
\end{remark}

There is, finally, an excision result that makes $R(N)$ a candidate for the UMAP graph, as follows:

\begin{theorem}\label{th 27}
  The filtered complex map $\phi: R(N) \subset V(Z,D)$ induces isomporphisms
  \begin{equation*}
    \phi_{\ast}: \pi_{0}R_{s}(N) \xrightarrow{\cong} \pi_{0}V_{s}(Z,D)
    \end{equation*}
for all $s \geq 0$. 
\end{theorem}

\begin{proof}
  The proof is similar to that of Theorem \ref{th 23}.
  
The map $\phi$ is the identity on vertices, so the functions $\phi_{\ast}$ are surjective.
  
  Suppose that there is a $1$-simplex $\{z,w\}$ of $V(Z,D)_{s}$. Then there is a path
  \begin{equation*}
    z=x_{0},x_{1}, \dots ,x_{p}=w
  \end{equation*}
  through rays $\{x_{i},x_{i+1}\}$ such that
  \begin{equation*}
    \sum_{i=0}^{p}\ d(x_{i},x_{i+1}) \leq s.
  \end{equation*}
  by Lemma \ref{lem 24}.
  But then $d(x_{i},x_{i+1}) \leq s$ for all $i$, so that $z$ and $w$ are in the same path component of the simplicial set $R(N)_{s}$.

  It follows that the functions $\phi_{\ast}$ are injective.
\end{proof}

\section{Weighted directed graphs}

A weighted directed graph $\Gamma$ consists of a set of edges $e:x \to y$, such that each edge $e$ has a weight $w(e) > 0$. For the present discussion, the vertices of $\Gamma$ are faces of the edges. I write $Z$ for the set of vertices of $\Gamma$.

Trivial examples are given by $1$-skeleta $\sk_{1}K$ of oriented simplicial complexes $K$, with weights $w(e) = 1$ for each $1$-simplex $e: x \to y$, and such that every vertex is a face of some non-degenerate $1$-simplex $e$.
\medskip

A weighted directed graph $\Gamma$ is said to be {\bf sparse} if all vertices $x$ of $\Gamma$ have low valence. This means that each  vertex of $\Gamma$ is in the boundary of a small (i.e. computable) number of edges.

Suppose that $x$ is a vertex of a transfer graph $\Gamma$.
A {\bf path} $P: x \dashrightarrow y$ from $x$ to another vertex $y$ in $\Gamma$ is a string of edges
\begin{equation}\label{eq 5}
  P:\ x=y_{0} \xrightarrow{e_{1}} y_{1} \xrightarrow{e_{2}} \dots \xrightarrow{e_{p}} y_{p} = y.
  \end{equation}
Say that the integer $p$ is the length $\ell(P)$ of the path $P$.

\begin{remark}
The collection of all paths $P: x \dashrightarrow y$ in the graph $\Gamma$ form a weighted graph $P(\Gamma)$ having the same vertices as the graph $\Gamma$.
The paths $P: x \dashrightarrow y$ and $Q: y \dashrightarrow z$ are composeable: the concatenation of $P$ with $Q$ defines a path $P \circ Q: x \dashrightarrow z$. Thus, $P(\Gamma)$ has more structure: $P(\Gamma)$ is the free category on the graph $\Gamma$.

The path graph $P(\Gamma)$ is not sparse in general.
\end{remark}

The weight $w(P)$ of the path $P$ can be defined by
\begin{equation}\label{eq 6}
  w(P) = \min_{i}\ \{w(e_{i})\}.
\end{equation}

\begin{remark}
  The definition of the weight of a path is somewhat arbitrary, and depends on applications. The assignment of (\ref{eq 6}) is motivated by graphs of data transfers, which are discussed below. One could, alternatively, set
  \begin{equation*}
    w(P) = \sum_{i}\ w(e_{i}).
    \end{equation*}
\end{remark}

Fix a positive integer $k$.
\medskip

The {\bf neighbourhood} $N_{k}(x)$ is the collection of all vertices $y$, which appear in paths
\begin{equation*}
Q:\  z_{0} \to z_{1} \to \dots \to z_{p},
\end{equation*}
having length $p=\ell(Q) \leq k$, such that $x=z_{i}$ for some $i$
\medskip

We assign a weight (or distance) $d(x,y)$ for all $y$ in the neighbourhood $N_{k}(x)$.

For $y \in N_{k}(x)$, define the {\bf weight sum} $\Sigma(x,y)$ by
\begin{equation*}
  \Sigma(x,y) = (\sum_{P: x \dashrightarrow y,\ell(P) \leq k}\ w(P)) + (\sum_{Q: y \dashrightarrow x,\ell(P) \leq k}\ w(Q)).
    \end{equation*}

In a graph of transactions, the weight sum $\Sigma(x,y)$ represents the total value of all transactions between $x$ and $y$. If $\Sigma(x,y)$ has a large value, then there is more business between $x$ and $y$, and these objects should be closer in some sense. To express this relationship, use the Shannon information function to define a distance
    \begin{equation}\label{eq 7}
      d(x,y) = e^{-\Sigma(x,y)}
    \end{equation}
    for $y \in N_{k}(x)$.

    Other approaches to defining a distance $d(x,y)$ for the vertices $y$ of $N_{k}(x)$ are certainly possible.
    \medskip
    
    We end up with a computable neighbourhood $N_{k}(x)$ of vertices in a sparse directed graph $\Gamma$ for each of its vertices $x$, with distances (weights) $d(x,y)$ for $y \in N_{k}(x) - \{x\}$.

    These are the inputs for the UMAP construction, which is described in Section 4.

\begin{example}[{\bf Data transfers}]\label{ex 30}
  A data transfer $e: x \to y$ from a computer account $x$ to a different account $y$ has a weight $w(e)$, which is the number of bytes transferred. The transfer $e$ also has source and target time stamps, $s(e)$ and $t(e)$, respectively, with $s(e) < t(e)$. Thus (provisionally), a graph $\Gamma$ of data transfers has edges $e: (x,s(e)) \to (y,t(e))$ with $x \ne y$, and its vertices consist of pairs $(x,t)$, where $x$ is a computer account and $t$ is either a source or a target timestamp for some edge.

  There may be multiple vertices $(x,t)$ for a fixed account $x$. Suppose that $t_{0} < t_{2} < \dots < t_{p}$ are the timestamps for a fixed account $x$. Say that the list $\{t_{0}, \dots ,t_{p}\}$ is the simplex of timestamps for the account $x$.

  For $x \ne y$, an edge $E: (x,s) \to (y,t)$ of the transfer graph $\Gamma$ consists of a transfer $e: (x,s(e)) \to (y,t(e))$, together with relations $s \leq s(e)$ and $t(e) \leq t$ in the simplices of timestamps for the accounts $x$ and $y$, respectively. The weight $w(E)$ is the weight $w(e)$ of the transfer $e$. The set vertices $Z$ of $\Gamma$ consists of all pairs $(x,s)$ of accounts $x$ and timestamps $s$ of transfers.
  
If all timestamps lie within a small enough interval, then the transfer graph $\Gamma$ is sparse.

This example motivates the definitions of weights of paths and distances within neighbourhoods that are seen above.
\smallskip

Explicitly, a path
   \begin{equation*}
    P:\ (x,s) = x_{0} \xrightarrow{E_{1}} x_{1} \xrightarrow{E_{2}} \dots \xrightarrow{E_{p}} x_{p} = (y,t)
  \end{equation*}
in $\Gamma$ consists of edges
  \begin{equation*}
    E_{i}: x_{i} = (x_{i},s_{i}) \to (x_{i+1},s_{i+1}) = x_{i+1}
  \end{equation*}
with $x_{i} \ne x_{i+1}$, and each such edge has weight $w(E_{i})$.

The weight $w(P)$ of the path $P$ is defined by
  \begin{equation*}
    w(P) = \min_{i} \{w(E_{i})\},
  \end{equation*}
  as in (\ref{eq 6})).
  The weight $w(P)$ represents the maximum amount of data that could be transferred from $(x,s)$ to $(y,t)$ along the path $P$.

Fix a positive integer $k$ and an element $x=(x,s)$ in the transfer graph $\Gamma$, and define the neighbourhood $N_{k}(x)$ as vertices of paths crossing $x$ of length at most $k$.

The weight sum $\Sigma(x,y)$ for $y \in N_{k}(x)$ is defined by
\begin{equation*}
  \Sigma(x,y) = (\sum_{P: x \dashrightarrow y,\ell(P) \leq k}\ w(P)) + (\sum_{Q: y \dashrightarrow x,\ell(P) \leq k}\ w(Q)),
\end{equation*}
and the weight $d(x,y)$ of the ray $\{x,y\}$ has the form
\begin{equation*}
  d(x,y) = e^{-\Sigma(x,y)}.
  \end{equation*}
  \end{example}

\begin{remark}[{\bf Undirected graphs}]\label{rem 31}
  The directed structure for the graph $\Gamma$ is a central feature of the examples discussed above. Analogous local to global methods apply equally well to construct ep-metric spaces and UMAP complexes for undirected graphs.
  \medskip
  
  Suppose that $\Omega$ is a sparse weighted graph, with weights $w(e)$ for the edges $e$ of $\Omega$. One assumes that the vertices of $\Omega$ are faces of its edges.

  Suppose that $x$ is a vertex of $\Omega$. Say that $y \in N_{k}[x]$ if there is a path (path), or string of edges
  \begin{equation*}
   P:\ x=x_{0} \overset{e_{1}}{\leftrightarrow} x_{1} \overset{e_{2}}{\leftrightarrow} \dots \overset{e_{p}}{\leftrightarrow} x_{p}=y
  \end{equation*}
  with $p \leq k$.
  
  Again there are choices, but define the weight $w(P)$ of the path $P: x \leftrightarrow y$ by
\begin{equation*}
  w(P) = \min_{i}\ \{w(e_{i})\}.
  \end{equation*}

Fix a vertex $x$ and a positive integer $k$. Define $N_{k}(x)$ to be the set of all vertices of $\Omega$ which lie on paths of length $\ell(P)$ at most $k$ that pass through $x$.

Write
  \begin{equation*}
   \Sigma(x,y) = \sum_{P: x \dashrightarrow y, \ell(P) \leq k}\ w(P),
  \end{equation*}
  and set
\begin{equation*}
  d(x,y) = e^{-\Sigma(x,y)}
\end{equation*}
for $y \in N_{k}(x)$.

One uses the weights $d(x,y)$ to construct an ep-metric on the neighbourhood $N_{k}(x)$. These ep-metrics patch together, to give an ep-metric on the full set $Z$ of vertices of $\Omega$.
\end{remark}

\section{Bags of words}

In the ``bag of words'' model for natural language processing (see, for example \cite{bleiLDA}), one starts with a collection $C=\{C_{1}, \dots ,C_{N}\}$  of {\bf documents} $C_{i}$, where each $C_{i} = (t_{i,1}, \dots ,t_{i,M_{i}})$ is a sequence of {\bf tokens} (ie. words, phrases, etc.), with possible repetitions.  The sequence $C$ is the {\bf corpus}.

The sequence $C_{i}$ is a function
\begin{equation*}
  C_{i}: \underline{M}_{i} \to \mathcal{T},
\end{equation*}
where $\mathcal{T}$ is the set of distinct tokens in all $C_{i}$, and $\underline{M}_{i} = \{1,2, \dots ,M_{i}\}$. The sequence $C_{i}$ may have repeats, so the function $C_{i}$ is not injective in general.

The usual thing is to amalgamate some tokens (by root words, or whatever), to form a surjective  map $\ell: \mathcal{T} \to \mathcal{L}$. The set $\mathcal{L}$ is the {\bf vocabulary} and its elements are called {\bf words}.

Write $p$ for the composite function
\begin{equation*}
  p: \sqcup_{i}\ \underline{M}_{i} \xrightarrow{C} \mathcal{T} \xrightarrow{\ell} \mathcal{L},
  \end{equation*}
and let $p_{i}: \underline{M}_{i} \to \mathcal{L}$ be the restriction of $p$ to the summand $\underline{M}_{i}$.

We assume that there are no common tokens (``stop words'') or rare tokens in the set $\mathcal{T}$, however these are determined. This means that the fibres $p^{-1}(w)$ of the function $p$ are neither too large nor too small, and in particular are computationally manageable. The function $p$ and its fibres are the objects
of interest for this discussion.

The fibres $p^{-1}(w)$ are the instances of the word $w \in \mathcal{L}$ in the corpus $C$.

\begin{remark}
In more generality, we could have functions $p_{i}: \underline{M}_{i} \to Z$ which cover a set $Z$, in the sense that the amalgamated function
\begin{equation*}
  p: \sqcup_{i}\ \underline{M}_{i} \to Z
\end{equation*}
is surjective. Here, the restriction of $p$ to the summand $\underline{M}_{i}$ is $p_{i}$. One assumes that the fibres $p_{i}^{-1}(z)$ for $z \in Z$ are computationally manageable (or tractable in the sense of the next section), as is the collection of functions $\{p_{i}\}$.

Subject to size assumptions on the cardinals $\underline{M}_{i}$ and the collection of functions $p_{i}$, the following discussion can be applied in such a setting.

One could even replace the sets $\underline{M}_{i}$ with metric spaces in the discussion that follows.
\end{remark}

Write $\mathcal{L}_{i}$ for the image of the restricted function
\begin{equation*}
 p_{i} = \ell_{i}: \underline{M}_{i} \xrightarrow{C_{i}} \mathcal{T} \xrightarrow{\ell} \mathcal{L}.
\end{equation*}
The composite $p_{i}$ restricts to a surjective function $p_{i}: \underline{M}_{i} \to \mathcal{L}_{i}$, and there is a commutative diagram of functions
\begin{equation*}
  \xymatrix{
    \underline{M}_{i} \ar[r]^{p_{i}} \ar[d] & \mathcal{L}_{i} \ar[d] \\
    \sqcup_{i}\ \underline{M}_{i} \ar[r]_{p} & \mathcal{L}
  }
\end{equation*}
in which the vertical maps are inclusions.
\medskip

Each set $\underline{M}_{i}$ has a metric $d$ with $d(x,y) = \vert y-x \vert$.
\medskip
    
Suppose that $r$ is a positive integer. Fix a word $v \in \mathcal{L}$, and suppose
 that $p^{-1}_{i}(w)_{\leq r}$ is the set of all elements $y \in p_{i}^{-1}(w)$ such that $d(x,y) \leq r$ for some $x \in p_{i}^{-1}(v)$. Then we have
    \begin{equation*}
    p_{i}^{-1}(w)_{\leq r} = p_{i}^{-1}(w) \cap (\cup_{x \in p^{-1}(v)}\ [x-r,x+r])
    \end{equation*}
    in the set $\underline{M}_{i}$.
    The subsets $p^{-1}_{i}(w)_{\leq r}$ filter the fibre $p_{i}^{-1}(w)$. Observe that $p_{i}^{-1}(v)_{\leq r} = p_{i}^{-1}(v)$.
\medskip    

Set
 \begin{equation}\label{eq 8}
d_{i}[r](v,w) = \sum_{x \in p_{i}^{-1}(v),y \in p_{i}^{-1}(w), d(x,y) \leq r}\ d(x,y),
  \end{equation}
and define
\begin{equation*}
  d[r](v,w) = \sum_{i}\ d_{i}[r](v,w)
  \end{equation*}
for all $v,w \in \mathcal{L}$.

The number $d_{i}[r](v,w)$ is non-zero if and only if there are elements $x \in p_{i}^{-1}(v)$ and $y \in p_{i}^{-1}(w)$ such that $d(x,y) \leq r$, and $d[r](v,w) \ne 0$ if and only if $d_{i}[r](v,w) \ne 0$ for some $i$.

In particular, $d_{i}[r](v,v)$ is the sum of the distances $d(x,y)$ between $x,y \in p_{i}^{-1}(v)$ such that $d(x,y) \leq r$, and $d_{i}[0](v,v)=0$.
It follows that $d[r](v,v)$ can be non-trivial for $r > 0$, and $d[0](v,v) = 0$.
\medskip

Take all elements $x$ of the fibres $p_{i}^{-1}(v)$ and form all intervals $[x-r,x+r]$ in $\underline{M}_{i}$. The union
    \begin{equation}\label{eq 9}
      N_{v}[r] = \cup_{i}\ (\cup_{x \in p_{i}^{-1}(v)}\ p_{i}[x-r,x+r]) \subset \mathcal{L}.
      \end{equation}
    is a neighbourhood of $v$ in $\mathcal{L}$.

    Observe that $N_{v}[0] = \{v\}$. Also, $N_{v}[r] \subset N_{v}[s]$ for $r \leq s$, and $\cup_{r}\ N_{v}[r] = \mathcal{L}$, so that the subsets $N_{v}[r]$ filter the set of words $\mathcal{L}$.
    \medskip
    
   Subject to fixing a positive integer $r$, the set $N_{v}[r]$ is a neighbourhood for $v \in \mathcal{L}$, and the number $d[r](v,w)$ is the weight of $w \in N_{v}[r]$.

   As in Section 5, the UMAP construction assembles the weighted neighbourhoods $(N_{v}[r],d[r])$, $v \in \mathcal{L}$, to form a the UMAP complex $V(\mathcal{L},N[r])$, an ep-metric space $(\mathcal{L},D[r])$, and a ray complex $R(N[r]) \subset V(\mathcal{L},D[r])$, all of which compute the same clusters.
   %--- see Remark \ref{rem 27}. 
     
\section{Sampling}

Suppose that the universal data set $\mathcal{U}$ has an ep-metric space structure, but with no other information.

In this case, one approximates (or discovers) a neighbourhood $N_{x}$ for a given point $x \in \mathcal{U}$ with a brute force method that is based on sampling techniques and construction of $k$-complete neighbourhoods within samples.
\medskip
 
Suppose that $Z$ is a randomly chosen subset of $\mathcal{U}$ (a sample), and that $Z$ is tractable in the sense that there is a cardinality bound
$\vert Z \vert \leq M$, where data sets of size at most $M$ can be analyzed by available computational devices.
We assume that $x \in Z$.

For such a subset $Z$ the distance function $d_{x}: Z \to [0,\infty)$, with $d_{x}(z) = d(x,z)$, can be computed, and the image $d_{x}(Z)$ of $d_{x}$ defines a tractable subset of the interval $[0,\infty)$. The set $Z$ is a disjoint union of fibres
    \begin{equation*}
      Z = \sqcup_{s \in d_{x}(Z)}\ d_{x}^{-1}(s)
    \end{equation*}
    of the distance function $d_{x}$.

    Suppose that $k$ is a fixed choice of positive integer with $k \leq \vert Z \vert$.
    
    The element $x$ has a uniquely defined $k$-complete neighbourhood $N$ in $Z$, as in Section 2, which is the smallest complete neighbourhood $Z(x,s)$ such that $\vert Z(x,s) \vert \geq k$. The neighbourhood $N$ is a union of fibres $d_{x}^{-1}(t)$ for smallest values of $t$, 

    This construction can be repeated, in parallel, for an appropriately sized collection of samples $Z_{1}, \dots ,Z_{p}$ that contain $x$, with distance functions $d_{x}: Z_{i} \to [0,\infty)$. Each sample $Z_{i}$ has a uniquely defined (and computable) $k$-complete neighbourhood $N_{i}$ of $x$, and the $k$-complete neighbourhood $N$ of $x$ in the union $\cup_{i}\ Z_{i}$ is a $k$-complete neighbourhood of $x$ in the smaller object $\cup_{i}\ N_{i}$, by Lemma \ref{lem 10}. 
      \medskip
      
There are various ways to invoke the samples $Z_{i}$:
\medskip

\noindent
1)\ Starting with a $k$-complete neighbourhood $N_{x}$ of $x$ in a sample $Z_{x}$, choose samples $Z_{y}$ for each $y \in N_{x}$, with associated $k$-complete neighbourhoods $N_{y} \subset Z_{y}$. The union $\cup_{y \in N_{x}}\ N_{y}$ contains a $k$-complete neighbourhood $N'_{x}$, which is potentially a better approximation of a $k$-complete neighbourhood of $x$ in the universe $\mathcal{U}$.

This sequence of steps is an analogue of the $k$-nearest neighbour algorithm of \cite{DML}.
\medskip

\noindent
2)\
The determination of a $k$-complete neighbourhood $N$ of $x$ in $V$ for some $V \subset \mathcal{U}$ can be extended to larger subsets of $\mathcal{U}$, subject to computational constraints, by adding more tractable samples to $V$. This is again a simple application of Lemma \ref{lem 10}.
\medskip

\noindent
3)\ If $Z_{1}, \dots ,Z_{p}$ is a tractable collection of tractable samples in $\mathcal{U}$, then we can find a $k$-complete neighbourhood $N_{y}$ in $Z = \cup_{i}\ Z_{i}$ for any $y \in Z$. The corresponding subcomplexes $V(N_{y}) \subset V(Z)$ and $R(N_{y}) \subset V(N_{y})$ determine filtered subcomplexes
      \begin{equation*}
        \cup_{y \in Z}\ R(N_{y}) \subset \cup_{y \in Z}\ V(N_{y}),
      \end{equation*}
      which lead to a UMAP-style analysis that computes the clusters of $V(Z)$, and approximates the clusters of all $V(\mathcal{U})$.
\medskip
 
    The sampling technique displayed here is completely brute force. It only approximates clusters and neighbouhoods of points, and does not speak to the entire data set $\mathcal{U}$.
    
    The method can be refined in the presence of global constraints, such as the local uniformity assumption of \cite{DML} that produces sets of $k$-nearest neighbours up to a probability estimate.

 \bigskip
%\vfill\eject

%\nocite{clusters} %Stable components and layers
%\nocite{fuzzy-p} %Fuzzy sets and presheaves
%\nocite{data-htpy} %Data and homotopy types
%\nocite{persist-htpy} %Persistent homotopy theory
%\nocite{dir-persist} %Directed persistence
%\nocite{br-pts} %Branch points and stability

%\nocite{BPN-entropy}
%\nocite{sym11030386}
%\nocite{Li2015}
%\nocite{fuzzy-Spivak}
%\nocite{HMc-2018}
%\nocite{Scocc-thesis}
%\nocite{Shannon48}
%\nocite{Belkin}
%\nocite{WS}

%\nocite{CMS} %mapper algorithm
%\nocite{metric} %Metric spaces and homotopy types
%\nocite{UMAP-stab} %Stability for UMAP
%\nocite{DCL} %k-nearest neighbour algorithm
%\nocite{TIMC-2020} %Labelled graph embeddings
%\nocite{HMc-2020} &UMAP

\nocite{github}

\bibliographystyle{plain}
\bibliography{spt}

\end{document}